\documentclass[11pt]{amsart}

\scrollmode

\usepackage{times}
\usepackage{amsmath,graphicx}
\usepackage{latexsym}
\usepackage{amscd,amsthm,amssymb}
\usepackage[all]{xy}

\newtheorem{thm}{Theorem}[section]
\newtheorem{prop}[thm]{Proposition}
\newtheorem{lem}[thm]{Lemma}

\theoremstyle{definition}
\newtheorem{ex}[thm]{Example}
\newtheorem{defn}[thm]{Definition}

\theoremstyle{remark}
\newtheorem{rem}[thm]{Remark}

\numberwithin{equation}{section}

\title{Bi-Lagrangian structures on nilmanifolds}
\author{M.~J.~D.~Hamilton}
\address{      Fachbereich Mathematik\\
               Universit\"at Stuttgart\\
               Pfaffenwaldring 57\\
               70569 Stuttgart\\
               Germany}
\email{mark.hamilton@math.lmu.de}
\date{\today}

\begin{document}

\begin{abstract} We study bi-Lagrangian structures (a symplectic form with a pair of complementary Lagrangian foliations, also known as para-K\"ahler or K\"unneth structures) on nilmanifolds of dimension less than or equal to $6$. In particular, building on previous work of several authors, we determine which $6$-dimensional nilpotent Lie algebras admit a bi-Lagrangian structure. In dimension $6$, there are (up to isomorphism) $26$ nilpotent Lie algebras which admit a symplectic form, $16$ of which admit a bi-Lagrangian structure and $10$ of which do not. We also calculate the curvature of the canonical connection of these bi-Lagrangian structures. 
\end{abstract}

\maketitle

\section{Introduction}
In this article we are interested in bi-Lagrangian structures on smooth manifolds, defined as follows:
\begin{defn}
Let $M$ be a smooth manifold. Then a {\em bi-Lagrangian structure} consists of a symplectic form $\omega\in\Omega^2(M)$ and a pair $\mathcal{F},\mathcal{G}$ of Lagrangian foliations on $M$ such that $TM=T\mathcal{F}\oplus T\mathcal{G}$. Bi-Lagrangian structures are also known as {\em K\"unneth structures}.
\end{defn}
Recall that according to the Frobenius theorem a {\em foliation} $\mathcal{F}$ on a smooth manifold $M$ is given by a distribution $T\mathcal{F}\subset TM$ which is integrable in the sense that $[X,Y]\in \Gamma(T\mathcal{F})$ for all sections $X,Y\in \Gamma(T\mathcal{F})$. A foliation $\mathcal{F}$ on a symplectic manifold $(M^{2n},\omega)$ is called {\em Lagrangian} if $T\mathcal{F}$ is a Lagrangian subbundle of $(TM,\omega)$, i.e.~the rank of $T\mathcal{F}$ is $n=\frac{1}{2}\dim M$ and $\omega(X,Y)=0$ for all $X,Y\in T\mathcal{F}$. 

In the following we will consider foliations as integrable distributions and place less emphasis on foliation-specific aspects such as the global topology or the structure of leaves.

For background on bi-Lagrangian structures see, for instance, \cite{Bo}, \cite{ES}, \cite{He} and the forthcoming book \cite{HK}. Bi-Lagrangian structures appear frequently in other contexts of geometry, for example, hypersymplectic structures, introduced by Hitchin \cite{Hit}, and Anosov symplectomorphisms define bi-Lagrangian structures (see Chapters 5 and 8 in \cite{HK}).
\begin{ex}\label{ex:standard bi-Lag}
The simplest example of a bi-Lagrangian structure is given by $M=\mathbb{R}^{2n}=\mathbb{R}^n\times\mathbb{R}^n$ with coordinates $x_1,\ldots,x_n,y_1,\ldots,y_n$, symplectic form 
\begin{equation*}
\textstyle\omega_0=\sum_{i=1}^ndx_i\wedge dy_i
\end{equation*}
and Lagrangian foliations $\mathcal{F}_0$ and $\mathcal{G}_0$, whose leaves are given by the affine Lagrangian subspaces $\mathbb{R}^n\times\{\ast\}$ and $\{\ast\}\times\mathbb{R}^n$. This bi-Lagrangian structure descends to the torus $T^{2n}$.
\end{ex}
Non-trivial examples of bi-Lagrangian structures are usually difficult to construct. An interesting class of examples comes from nilmanifolds. We want to study left-invariant bi-Lagrangian structures on Lie groups $G$, which we can equivalently think of as the following linear structures on the Lie algebra $\mathfrak{g}$ of $G$:
\begin{defn}
Let $\mathfrak{g}$ be a real Lie algebra. A {\em symplectic form} on $\mathfrak{g}$ is a closed and non-degenerate $2$-form $\omega \in \Lambda^2\mathfrak{g}^*$ and a {\em foliation} on $\mathfrak{g}$ is a subalgebra $\mathcal{F}\subset\mathfrak{g}$. A {\em bi-Lagrangian structure} on $\mathfrak{g}$ consists of a symplectic form $\omega$ and two Lagrangian subalgebras $\mathcal{F},\mathcal{G}\subset\mathfrak{g}$ such that $\mathfrak{g}=\mathcal{F}\oplus \mathcal{G}$ (vector space direct sum).
\end{defn}
Bi-Lagrangian structures on Lie algebras in general have been studied, for example, in \cite{Bai}, \cite{BB} and \cite{Ka}.

In the case of a nilpotent Lie algebra $\mathfrak{g}$, the corresponding left-invariant bi-Lagrangian structure on the associated, simply connected nilpotent Lie group $G$ induces a bi-Lagrangian structure on compact nilmanifolds $G/\Gamma$, where $\Gamma\subset G$ is a lattice, acting by left-multiplication on $G$. According to a theorem of Malcev \cite{Mal} a simply connected nilpotent Lie group $G$ has a lattice if and only if its Lie algebra $\mathfrak{g}$ admits a basis with rational structure constants. This is the case in all examples that we consider.

In this paper we discuss the existence of bi-Lagrangian structures on symplectic nilpotent Lie algebras $\mathfrak{g}$ of dimension $2$, $4$ and $6$. Existence of a bi-Lagrangian structure is shown by exhibiting a specific example of a symplectic form $\omega$ on $\mathfrak{g}$ and two complementary Lagrangian subalgebras $\mathcal{F},\mathcal{G}\subset\mathfrak{g}$. Proving non-existence of a bi-Lagrangian structure (for any symplectic form on $\mathfrak{g}$) is more involved. 
The main result can be summarized as follows:
\begin{thm}
\begin{enumerate}
\item In dimension $2$ there is a single nilpotent Lie algebra up to isomorphism. It admits a bi-Lagrangian structure.
\item In dimension $4$ there are three nilpotent Lie algebras up to isomorphism, each of which admits a symplectic form. Two of them admit a bi-Lagrangian structure and one of them does not (see Table \ref{table:4-dim} in the appendix).
\item In dimension $6$ there are $26$ nilpotent Lie algebras up to isomorphism that admit a symplectic form. $16$ of them admit a bi-Lagrangian structure and $10$ of them do not (see Table \ref{table:symp}).
\end{enumerate}
\end{thm}
\begin{rem} The general existence question for a single Lagrangian foliation on nilpotent Lie algebras has been settled before by Baues and Cort\'es: According to Corollary 3.13.~in \cite{BC} every symplectic form on a nilpotent Lie algebra admits a Lagrangian foliation.
\end{rem}
It is an interesting observation, originally perhaps due to H.~Hess \cite{He}, that every bi-Lagrangian structure $(\omega,\mathcal{F},\mathcal{G})$ on a smooth, $2n$-dimensional manifold $M$ defines an associated canonical affine connection $\nabla$ on $TM$, see Section \ref{sect:prep} for the definition. It has the following properties:
\begin{enumerate}
\item $\nabla$ preserves both subbundles $T\mathcal{F}$ and $T\mathcal{G}$
\item $\nabla$ is symplectic, i.e.~$\omega$ is parallel with respect to $\nabla$
\item $\nabla$ is torsion-free: $\nabla_XY-\nabla_YX=[X,Y]$ for all $X,Y\in\mathfrak{X}(M)$.
\end{enumerate}
\begin{rem}\label{rem:canonical charact}
The canonical connection of a bi-Lagrangian structure is uniquely characterized by these three properties.
\end{rem}
The canonical connection coincides with the Levi-Civita connection of a certain pseudo-Riemannian metric $g$ associated to the bi-Lagrangian structure: We denote the projections of a vector field $X$ onto $T\mathcal{F}$ and $T\mathcal{G}$ by $X_F$ and $X_G$ and define
\begin{equation*}
I\colon TM\longrightarrow TM,\quad X_F+X_G\longmapsto X_F-X_G
\end{equation*}
and
\begin{equation*}
g\colon TM\times TM\longrightarrow\mathbb{R},\quad (X,Y)\longmapsto \omega(IX,Y).
\end{equation*}
Then $I^2=\mathrm{Id}_{TM}$ and $g$ is a pseudo-Riemannian metric of neutral signature $(n,n)$ on the manifold $M^{2n}$. The pair $(I,g)$ is called a {\em para-K\"ahler structure} on $M$. It turns out that the canonical connection of the bi-Lagrangian structure $(\omega,\mathcal{F},\mathcal{G})$ is equal to the Levi-Civita connection of $g$; cf.~\cite{CFG, ES}.

We are interested especially in the curvature of the canonical connection. It is easy to see that for the standard bi-Lagrangian structure $(\omega_0,\mathcal{F}_0,\mathcal{G}_0)$ on $\mathbb{R}^{2n}$, considered in Example \ref{ex:standard bi-Lag}, the curvature tensor $R$ vanishes identically. There is a Darboux-type converse to this statement: The curvature of the canonical connection of any bi-Lagrangian structure vanishes if and only if it is locally isomorphic to the standard structure $(\omega_0,\mathcal{F}_0,\mathcal{G}_0)$.

We calculate the curvature of the canonical connection for all our examples and show:
\begin{thm}
\begin{enumerate}
\item In dimension $2$ and $4$ the canonical connection is flat for all our examples of bi-Lagrangian structures.
\item In dimension $6$ the canonical connection is flat for $8$ examples and non-flat for the remaining $8$ examples of bi-Lagrangian structures. All our examples of bi-Lagrangian structures are Ricci-flat (see Table \ref{table:curvature}).
\end{enumerate}
\end{thm}
Examples of non-flat, Ricci-flat para-K\"ahler structures on nilpotent Lie algebras have been constructed before in \cite{R} and \cite{CR}. The latter reference contains a general study of the Ricci curvature of para-K\"ahler structures (and, more generally, almost para-Hermitian structures, such as nearly para-K\"ahler structures).

An analysis of the (Ricci) curvature of related structures, such as {\em tri-Lagrangian structures} studied in \cite{G, CS}, will be left for future research.

\subsection*{Notation}
All nilpotent Lie algebras that we consider are real. For a nilpotent Lie algebra $\mathfrak{g}$ we denote by $e_1,\ldots,e_n$ a basis and by $\alpha_1,\ldots,\alpha_n$ the dual basis. The differentials $d\alpha_i$ are related to the commutators $[e_j,e_k]$ by
\begin{equation*}
d\alpha_i(e_j,e_k)=-\alpha_i([e_j,e_k]).
\end{equation*}
We set $\alpha_{ij}=\alpha_i\wedge\alpha_j$. A foliation (subalgebra) $\mathcal{F}$ in a nilpotent Lie algebra is specified by a basis $\{f_1,\ldots,f_m\}$.

We follow the notation for Lie algebras in \cite{BM}: $A_n$ denotes an $n$-dimensional abelian Lie algebra and $L_m$ or $L_{m,k}$ an $m$-dimensional nilpotent Lie algebra.

In the calculation of the canonical connection $\nabla$ of a bi-Lagrangian structure only the non-vanishing components in the bases for the Lagrangian foliations are given. In the calculation of the curvature tensor $R$ all non-zero components up to the symmetry $R(X,Y)Z=-R(Y,X)Z$ are given.

\section{Preparations}\label{sect:prep}
Suppose that $\mathfrak{g}$ is a $4$- or $6$-dimensional nilpotent Lie algebra and we want to prove that $\mathfrak{g}$ does not admit a bi-Lagrangian structure. We will frequently use the following standard lemma.
\begin{lem}\label{lem:subalg}
Every subalgebra of a nilpotent Lie algebra is nilpotent. If $\mathfrak{h}$ is an $n$-dimensional nilpotent Lie algebra, then $\dim [\mathfrak{h},\mathfrak{h}]\leq n-2$. In particular, every $2$-dimensional nilpotent Lie algebra is abelian and every $3$-dimensional nilpotent Lie algebra $\mathfrak{h}$ satisfies $[\mathfrak{h},[\mathfrak{h},\mathfrak{h}]]=0$.
\end{lem}
Lemma \ref{lem:subalg} restricts the possible vector subspaces in $\mathfrak{g}$ that are subalgebras and hence define foliations. Together with the condition that the symplectic form $\omega$ is closed and non-degenerate this can be used to rule out the existence of a bi-Lagrangian structure on $\mathfrak{g}$.

Suppose that $(\omega,\mathcal{F},\mathcal{G})$ is a bi-Lagrangian structure on a smooth manifold $M$. We want to define the associated canonical connection: For vector fields $X,Y\in\mathfrak{X}(M)$ let $D(X,Y)\in\mathfrak{X}(M)$ be the unique vector field defined by
\begin{equation*}
i_{D(X,Y)}\omega=L_Xi_Y\omega.
\end{equation*}
Then we set
\begin{align*}
\nabla_XY&=D(X_F,Y)_F+[X_G,Y]_F\quad\forall Y\in\Gamma(T\mathcal{F})\\
\nabla_XY&=D(X_G,Y)_G+[X_F,Y]_G\quad\forall Y\in\Gamma(T\mathcal{G}).
\end{align*}
One can check that these expressions define connections on the vector bundles $T\mathcal{F}$, $T\mathcal{G}$. The direct sum defines an affine connection $\nabla$ on $TM$, called the {\em canonical connection} or {\em K\"unneth connection}.

The curvature $R$ of the canonical connection is defined in the standard way by
\begin{equation*}
R(X,Y)Z=\nabla_X\nabla_YZ-\nabla_Y\nabla_XZ-\nabla_{[X,Y]}Z.
\end{equation*}
The curvature tensor satisfies the Bianchi identity
\begin{equation*}
R(X,Y)Z + R(Y, Z)X + R(Z, X)Y = 0\quad\forall X,Y,Z\in TM.
\end{equation*}
In addition, the tangent bundle $TM$ is flat along the leaves of both $\mathcal{F}$ and $\mathcal{G}$:
\begin{equation*}
R(X,Y)Z=0\quad\forall Z\in TM
\end{equation*}
whenever $X,Y\in T\mathcal{F}$ or $X,Y\in T\mathcal{G}$. Together with the Bianchi identity we get the symmetry
\begin{equation*}
R(X,Y)Z=R(X,Z)Y\quad\forall X\in TM
\end{equation*}
whenever $Y,Z\in T\mathcal{F}$ or $Y,Z\in T\mathcal{G}$.

Finally, the Ricci curvature $\mathrm{Ric}(X,Y)$ of vectors $X,Y\in TM$ is
defined as the trace of the map
\begin{equation*}
TM\longrightarrow TM,\quad Z\longmapsto R(Z, X)Y .
\end{equation*}

\section{Two- and four-dimensional nilpotent Lie algebras}

The case of dimension $2$ is trivial: there is a single nilpotent $2$-dimensional Lie algebra, the abelian Lie algebra $A_2$. A bi-Lagrangian structure is given by the symplectic form $\omega=\alpha_{12}$ and the Lagrangian foliations $\mathcal{F}=\{e_1\},\mathcal{G}=\{e_2\}$. The canonical connection $\nabla$ is trivial in the basis $e_1,e_2$ and its curvature $R$ vanishes.

We now consider the $4$-dimensional case, see Table \ref{table:4-dim} in the appendix. There are three nilpotent Lie algebras of dimension $4$, $A_4$, $L_3\oplus A_1$ and $L_4$, each of which is symplectic. The first two admit a bi-Lagrangian structure, while the third one does not. We prove non-existence for any symplectic form in the case of $L_4$ in the following subsection.

\subsection{Non-existence of a bi-Lagrangian structure on $L_4$}

\begin{lem}\label{lem:4}
Suppose that $L_4$ has two complementary $2$-dimensional subalgebras $\mathcal{F},\mathcal{G}$. Then one of the two subalgebras, say $\mathcal{F}$, has a basis of the form
\begin{equation*}
e_1+a_2e_2+a_3e_3,\quad e_4.
\end{equation*}
\end{lem}
\begin{proof}
It is clear that one of the subalgebras has a basis vector of the form
\begin{equation*}
\textstyle f_1=e_1+\sum_{i=2}^4a_ie_i.
\end{equation*}
Any other basis vector of the same subalgebra can be assumed to be of the form
\begin{equation*}
\textstyle f_2=\sum_{j=2}^4b_je_j.
\end{equation*}
We get
\begin{equation*}
[f_1,f_2] = -b_2e_3-b_3e_4.
\end{equation*}
Lemma \ref{lem:subalg} implies that $b_2=b_3=0$, hence the claim.
\end{proof}

\begin{prop}
The Lie algebra $L_4$ does not admit a bi-Lagrangian structure.
\end{prop}
\begin{proof}
Any closed $2$-form on $L_4$ is of the form
\begin{equation*}
\omega=\omega_{12}\alpha_{12}+\omega_{13}\alpha_{13}+\omega_{23}\alpha_{23}+\omega_{14}\alpha_{14}.
\end{equation*}
If the form $\omega$ is symplectic, then $\omega_{14}\neq 0$. It follows that the subalgebra $\mathcal{F}$ in Lemma \ref{lem:4} cannot be Lagrangian.
\end{proof}

\subsection{Calculation of the curvature of bi-Lagrangian structures in Table \ref{table:4-dim}}

We calculate the canonical connection and the curvature of the bi-Lagrangian structure for $A_4$ and $L_3\oplus A_1$. It turns out that both examples are flat.

\subsubsection{$A_4$}
A simple calculation shows that the canonical connection $\nabla$ is trivial in the basis $e_1,e_2,e_3,e_4$ and $R=0$. 

\subsubsection{$L_3\oplus A_1$}
\begin{equation*}
\nabla_{e_1}e_1=-e_3,\quad \nabla_{e_1}e_2=-e_4.
\end{equation*}
\begin{equation*}
R=0.
\end{equation*}

\section{Six-dimensional nilpotent Lie algebras}
We now consider the $6$-dimensional case. The symplectic nilpotent Lie algebras of dimension $6$ have been determined (independently) by \cite{S}, \cite{KGM} and \cite{BM}. The latter two references contain explicit symplectic forms. See Table \ref{table:compare} for a comparison of notation in these three references, Table \ref{table:structure} for the structure constants and Table \ref{table:symp} for symplectic forms and bi-Lagrangian structures.

\begin{rem}\label{rem:errors} The lists in \cite{BM} and \cite{KGM} for symplectic forms on $6$-dimensional nilpotent Lie algebras contain several errors: 
\begin{itemize}
\item The $2$-forms in \cite[Table 3]{BM} for the following Lie algebras are not symplectic:
\begin{equation*}
L_{6,1},L_{6,12},L_{6,15},L_{6,16},L_{6,17}^-.
\end{equation*}
\item The following $2$-forms in \cite[Section 3]{KGM} are not symplectic:
\begin{equation*}
16.\,\omega_2\quad\text{and}\quad 23.\,\omega_3.
\end{equation*}
Because of this error in \cite{KGM} we do not use the classification of symplectic forms (up to automorphisms of the Lie algebras) that is stated in this reference.
\end{itemize}
\end{rem}
There are in total $26$ nilpotent $6$-dimensional Lie algebras (up to isomorphism) which admit a symplectic form, $16$ of which admit a bi-Lagrangian structure and $10$ of which do not admit such a structure. If a bi-Lagrangian structure exists, an example is given in Table \ref{table:symp}. We now prove non-existence (for any symplectic form) in the remaining cases.

\subsection{Non-existence of bi-Lagrangian structures}
We will repeatedly use the following lemma.
\begin{lem}\label{lem:lie algebra subspace sum} Let $\mathfrak{g}$ be a $6$-dimensional Lie algebra with basis $e_1,\ldots,e_6$.
\begin{enumerate}
\item If a $3$-dimensional vector subspace $\mathcal{F}\subset\mathfrak{g}$ has a basis vector with a non-zero $e_1$-component, then a basis of $\mathcal{F}$ is of the form
\begin{equation*}
\textstyle f_1=e_1+\sum_{i=2}^6a_ie_i,\quad f_2=\sum_{j=2}^6b_je_j,\quad f_3=\sum_{k=2}^6c_ke_k,
\end{equation*}
with $a_i,b_j,c_k\in\mathbb{R}$.
\item If $\mathfrak{g}=\mathcal{F}\oplus\mathcal{G}$ is a sum of two $3$-dimensional vector subspaces, then at least one of $\mathcal{F},\mathcal{G}$ has a basis as in (a).
\end{enumerate}

\end{lem}

\subsubsection{$L_4\oplus A_2$}

\begin{lem}\label{lem:L4A2 symp}
A $2$-form $\omega$ on $L_4\oplus A_2$ is closed if and only if it is of the form
\begin{align*}
\omega&=\omega_{12}\alpha_{12}+\omega_{13}\alpha_{13}+\omega_{23}\alpha_{23}+\omega_{14}\alpha_{14}+\omega_{24}\alpha_{24}+\omega_{34}\alpha_{34}\\
&\quad+\omega_{15}\alpha_{15}+\omega_{25}\alpha_{25}+\omega_{16}\alpha_{16}.
\end{align*}
The form is symplectic if and only if $\omega_{16}\omega_{25}\omega_{34}\neq 0$.
\end{lem}

\begin{lem}\label{lem:sub L4A2} Suppose that $L_4\oplus A_2$ has two complementary $3$-dimensional subalgebras $\mathcal{F},\mathcal{G}$. Then one of the two subalgebras, say $\mathcal{F}$, has a basis of either one of the following forms:
\begin{enumerate}
\item $e_1+a_2e_2+a_3e_3+a_4e_4,\quad e_5+b_3e_3+b_4e_4,\quad e_6$
\item $e_1+a_2e_2+a_5e_5+a_6e_6,\quad e_3+b_6e_6,\quad e_4+c_6e_6$
\item $e_1+a_2e_2+a_4e_4+a_5e_5,\quad e_3+b_4e_4,\quad e_6$
\item $e_1+a_2e_2+a_3e_3+a_5e_5,\quad e_4,\quad e_6$.
\end{enumerate}
\end{lem}
\begin{proof}
Suppose that $\mathcal{F}$ has basis vectors $f_1,f_2$ as in Lemma \ref{lem:lie algebra subspace sum} (a). We get
\begin{equation}\label{eqn:L4A2}
[f_1,f_2]=-(b_2e_5+b_5e_6)
\end{equation}
and $[f_1,[f_1,f_2]]=b_2e_6$. Lemma \ref{lem:subalg} implies that $b_2=0$. A third basis vector is of the form
\begin{equation*} 
f_3=c_3e_3+c_4e_4+c_5e_5+c_6e_6.
\end{equation*}
There are two cases:
\begin{itemize}
\item $b_5\neq 0$ and $c_5=0$. Then equation \eqref{eqn:L4A2} shows that $e_6\in\mathcal{F}$. This results in a basis of the first form.
\item $b_5=0$ and $c_5=0$. Then the remaining two basis vectors are of the form
\begin{equation*}
f_2=b_3e_3+b_4e_4+b_6e_6,\quad f_3=c_3e_3+c_4e_4+c_6e_6.
\end{equation*}
If $b_3\neq 0$ and $c_3=0$ we get the bases in (b) and (c), depending on whether $c_4\neq 0$ or $c_4=0$. If $b_3=0$ and $c_3=0$ we get the basis in (d).
\end{itemize}
\end{proof}

\begin{prop} The Lie algebra $L_4\oplus A_2$ does not admit a bi-Lagrangian structure.
\end{prop}
\begin{proof} For the symplectic form as in Lemma \ref{lem:L4A2 symp} the subalgebra $\mathcal{F}$ in Lemma \ref{lem:sub L4A2} cannot be Lagrangian.
\end{proof}

\subsubsection{$L_{6,13}$}

\begin{lem}\label{lem:613symp} Every closed $2$-form on $L_{6,13}$ is of the form
\begin{align*}
\omega&=\omega_{12}\alpha_{12}+\omega_{13}\alpha_{13}+\omega_{23}\alpha_{23}+\omega_{14}\alpha_{14}+\omega_{24}\alpha_{24}-\omega_{16}\alpha_{34}\\
&\quad+\omega_{15}\alpha_{15}-\omega_{26}\alpha_{45}+\omega_{16}\alpha_{16}+\omega_{26}\alpha_{26}.
\end{align*}
If $\omega$ is symplectic, then $\omega_{26}$ has to be nonzero.
\end{lem}
\begin{lem}\label{lem:sub L613} Let $\mathcal{F}$ be a $3$-dimensional subalgebra of $L_{6,13}$ with a basis vector having a non-zero $e_1$-component. Then $\mathcal{F}$ has a basis of either one of the following two forms:
\begin{enumerate}
\item $e_1+a_2e_2+a_4e_4+a_5e_5,\quad e_3+b_5e_5,\quad e_6$
\item $e_1+a_2e_2+a_3e_3+a_4e_4,\quad e_5,\quad e_6$.
\end{enumerate}
\end{lem}
\begin{proof} Suppose that $\mathcal{F}$ has basis vectors $f_1,f_2$ as in Lemma \ref{lem:lie algebra subspace sum} (a). We get
\begin{equation}\label{eqn:L613}
[f_1,f_2]=-(b_2e_4+b_4e_5+(b_5+a_2b_3-a_3b_2)e_6)
\end{equation}
and $[f_1,[f_1,f_2]]=b_2e_5+b_4e_6$. Lemma \ref{lem:subalg} implies that $b_2=b_4=0$. There are two cases:
\begin{itemize}
\item $b_3=0$. Then either $b_5=0$, hence $b_6\neq 0$ and $e_6\in\mathcal{F}$. This results in a basis of type (a) or (b). Or $b_5\neq 0$ and equation \eqref{eqn:L613} shows that $e_6\in\mathcal{F}$. This results in a basis of type (b).
\item $b_3\neq 0$. A third basis vector of $\mathcal{F}$ is then of the form
\begin{equation*}
f_3=c_5e_5+c_6e_6.
\end{equation*}
If $c_5=0$, then $c_6\neq 0$ and $e_6\in\mathcal{F}$. This results in a basis of type (a). If $c_5\neq 0$, then the equation $[f_1,f_3]=-c_5e_6$ again shows that $e_6\in\mathcal{F}$. This case cannot occur, because then $\mathcal{F}$ had to be at least $4$-dimensional.
\end{itemize}
\end{proof}
\begin{prop} The Lie algebra $L_{6,13}$ does not admit a bi-Lagrangian structure.
\end{prop}
\begin{proof}
Suppose that $\mathcal{F},\mathcal{G}$ is a bi-Lagrangian structure. Considering the bases in Lemma \ref{lem:sub L613}, it follows that exactly one of $\mathcal{F},\mathcal{G}$ must have a basis with a non-zero $e_1$-component, say $\mathcal{F}$. Otherwise $\mathcal{F}$ and $\mathcal{G}$ cannot be complementary. 

In the first case $\mathcal{F}$ has a basis of the form
\begin{equation*}
e_1+a_2e_2+a_4e_4+a_5e_5,\quad e_3+b_5e_5,\quad e_6.
\end{equation*}
It follows that $\mathcal{G}$ has a basis of the form
\begin{align*}
g_1&=e_2+x_3e_3+x_5e_5+x_6e_6\\
g_2&=e_4+y_3e_3+x_5e_5+y_6e_6\\
g_3&=z_3e_3+z_5e_5+z_6e_6.
\end{align*}
We have $[g_1,g_2]=-y_3e_6$. This implies that $y_3=0$, since otherwise $[g_1,g_2]\in\mathcal{F}$. Similarly $[g_1,g_3]=-z_3e_6$, hence $z_3=0$. Then also $z_5\neq 0$, since otherwise $g_3$ is a multiple of $e_6\in\mathcal{F}$.

The basis is now
\begin{equation*}
g_1=e_2+x_3e_3+x_6e_6,\quad g_2=e_4+y_6e_6,\quad g_3=e_5+z_6e_6.
\end{equation*}
Considering the symplectic form $\omega$ as in Lemma \ref{lem:613symp} and the fact that $\omega_{26}\neq 0$, it follows that $\omega(g_2,g_3)\neq 0$, hence $\mathcal{G}$ cannot be Lagrangian. This is a contradiction.

In the second case $\mathcal{F}$ has a basis of the form
\begin{equation*}
e_1+a_2e_2+a_3e_3+a_4e_4,\quad e_5,\quad e_6.
\end{equation*}
It follows that $\mathcal{G}$ has a basis of the form
\begin{equation*}
g_1=e_2+x_5e_5+x_6e_6,\quad g_2=e_3+y_5e_5+y_6e_6,\quad g_3=e_4+z_5e_5+z_6e_6.
\end{equation*}
We have $[g_1,g_2]=-e_6\in\mathcal{F}$. This is a contradiction.
\end{proof}

\subsubsection{$L_{5,6}\oplus A_1$}

\begin{lem}\label{lem:sub L561} Let $\mathcal{F}$ be a $3$-dimensional subalgebra of $L_{5,6}\oplus A_1$ with a basis vector having a non-zero $e_1$-component. Then $\mathcal{F}$ has a basis of either one of the following two forms:
\begin{enumerate}
\item $e_1+a_2e_2+a_3e_3+a_4e_4,\quad e_5+b_3e_3,\quad e_6$
\item $e_1+a_2e_2+a_4e_4+a_5e_5,\quad e_3,\quad e_6$
\end{enumerate}
\end{lem}
\begin{proof}
Suppose that $\mathcal{F}$ has basis vectors $f_1,f_2$ as in Lemma \ref{lem:lie algebra subspace sum} (a). We get
\begin{align}
[f_1,f_2]&=-(b_2e_4+b_4e_5+(b_5+a_2b_4-a_4b_2)e_6)\label{eqn:L561}\\
[f_1,[f_1,f_2]]&=b_2e_5+b_4e_6+a_2b_2e_6.\nonumber
\end{align}
Lemma \ref{lem:subalg} implies that $b_2=b_4=0$. A third basis vector is of the form
\begin{equation*}
f_3=c_3e_3+c_5e_5+c_6e_6.
\end{equation*} 
There are two cases:
\begin{enumerate}
\item $b_5\neq 0$ and $c_5=0$. Then equation \eqref{eqn:L561} implies that $e_6\in\mathcal{F}$. This results in a basis of type (a).
\item $b_5=0$ and $c_5=0$. This results in a basis of type (b).
\end{enumerate}
\end{proof}

\begin{prop}
The Lie algebra $L_{5,6}\oplus A_1$ does not have two complementary $3$-dimensional subalgebras $\mathcal{F},\mathcal{G}$. In particular, $L_{5,6}\oplus A_1$ does not admit a bi-Lagrangian structure.
\end{prop}
\begin{proof}
Suppose that $\mathcal{F},\mathcal{G}$ are two complementary $3$-dimensional subalgebras of $L_{5,6}\oplus A_1$. Considering the bases in Lemma \ref{lem:sub L561}, it follows that exactly one of $\mathcal{F},\mathcal{G}$ must have a basis with a non-zero $e_1$-component, say $\mathcal{F}$. Otherwise $\mathcal{F}$ and $\mathcal{G}$ cannot be complementary. 

In the first case $\mathcal{F}$ has a basis of the form
\begin{equation*}
e_1+a_2e_2+a_3e_3+a_4e_4,\quad e_5+b_3e_3,\quad e_6.
\end{equation*}
It follows that $\mathcal{G}$ has a basis of the form
\begin{align*}
g_1&=e_2+x_3e_3+x_5e_5+x_6e_6\\
g_2&=z_3e_3+y_5e_5+y_6e_6\\
g_3&=e_4+x_3e_3+z_5e_5+z_6e_6.
\end{align*}
Then $[g_1,g_3]=-e_6\in\mathcal{F}$. This is a contradiction.

In the second case $\mathcal{F}$ has a basis of the form
\begin{equation*}
e_1+a_2e_2+a_4e_4+a_5e_5,\quad e_3,\quad e_6.
\end{equation*}
It follows that $\mathcal{G}$ has a basis of the form
\begin{equation*}
g_1=e_2+x_3e_3+x_6e_6,\quad g_2=e_4+y_3e_3+y_6e_6,\quad g_3=e_5+z_3e_3+z_6e_6.
\end{equation*}
Then $[g_1,g_2]=-e_6\in\mathcal{F}$. This is a contradiction. 
\end{proof}

\subsubsection{$L_{6,14}$}

\begin{lem}\label{lem:614symp} Every closed $2$-form on $L_{6,14}$ is of the form
\begin{align*}
\omega&=\omega_{12}\alpha_{12}+\omega_{13}\alpha_{13}+\omega_{23}\alpha_{23}+\omega_{14}\alpha_{14}+\omega_{24}\alpha_{24}-\omega_{16}\alpha_{34}\\
&\quad+\omega_{15}\alpha_{15}+\omega_{16}\alpha_{25}-\omega_{26}\alpha_{45}+\omega_{16}\alpha_{16}+\omega_{26}\alpha_{26}.
\end{align*}
If $\omega$ is symplectic, then at least one of $\omega_{16},\omega_{26}$ has to be non-zero.
\end{lem}

\begin{lem}\label{lem:sub L614} Let $\mathcal{F}$ be a $3$-dimensional subalgebra of $L_{6,14}$ with a basis vector having a non-zero $e_1$-component. Then $\mathcal{F}$ has a basis of either one of the following two forms:
\begin{enumerate}
\item $e_1+a_2e_2+a_4e_4+a_5e_5,\quad e_3+b_5e_5,\quad e_6$
\item $e_1+a_2e_2+a_3e_3+a_4e_4,\quad e_5,\quad e_6$.
\end{enumerate}
If $\mathcal{F}$ is Lagrangian, then $\omega_{26}\neq 0$ for the symplectic form in Lemma \ref{lem:614symp}.
\end{lem}
\begin{proof}
Suppose that $\mathcal{F}$ has basis vectors $f_1,f_2$ as in Lemma \ref{lem:lie algebra subspace sum} (a). We get
\begin{align}
[f_1,f_2]&=-(b_2e_4+b_4e_5+(b_5+a_2b_3-a_3b_2+a_2b_4-a_4b_2)e_6)\label{eqn:L614}\\
[f_1,[f_1,f_2]]&=b_2e_5+b_4e_6+a_2b_2e_6.\nonumber
\end{align}
Lemma \ref{lem:subalg} implies that $b_2=b_4=0$. There are two cases:
\begin{enumerate}
\item $b_3\neq 0$. Then a third basis vector of $\mathcal{F}$ is of the form
\begin{equation*}
f_3=c_5e_5+c_6e_6.
\end{equation*}
If $c_5=0$, then $c_6\neq 0$ and $e_6\in\mathcal{F}$. This results in a basis of type (a). If $c_5\neq 0$, then the equation $[f_1,f_3]=-c_5e_6$ implies again that $e_6\in\mathcal{F}$. This case cannot occur, because then $\mathcal{F}$ has to be at least $4$-dimensional.
\item $b_3=0$. If $b_5=0$, then $b_6\neq 0$, hence $e_6\in\mathcal{F}$. This results either in a basis of type (a) or (b). If $b_5\neq 0$, then equation \eqref{eqn:L614} shows that $e_6\in\mathcal{F}$. This results in a basis of type (b).
\end{enumerate}
\end{proof}

\begin{prop}
The Lie algebra $L_{6,14}$ does not have two complementary $3$-dimensional subalgebras $\mathcal{F},\mathcal{G}$. In particular, $L_{6,14}$ does not admit a bi-Lagrangian structure.
\end{prop}
\begin{proof}
Suppose that $L_{6,14}$ has two complementary $3$-dimensional subalgebras $\mathcal{F},\mathcal{G}$. Considering the bases in Lemma \ref{lem:sub L614}, it follows that exactly one of $\mathcal{F},\mathcal{G}$ must have a basis with a non-zero $e_1$-component, say $\mathcal{F}$. Otherwise $\mathcal{F}$ and $\mathcal{G}$ cannot be complementary.

In the first case $\mathcal{F}$ has a basis of the form
\begin{equation*}
e_1+a_2e_2+a_4e_4+a_5e_5,\quad e_3+b_5e_5,\quad e_6.
\end{equation*}
It follows that $\mathcal{G}$ has a basis of the form
\begin{align*}
g_1&=e_2+x_3e_3+x_5e_5+x_6e_6\\
g_2&=e_4+y_3e_3+y_5e_5+y_6e_6\\
g_3&=z_3e_3+z_5e_5+z_6e_6.
\end{align*}
We have $[g_1,g_2]=-e_6-y_3e_6$. Hence $y_3=-1$, otherwise $[g_1,g_2]$ is a multiple of $e_6\in\mathcal{F}$. Similarly, $[g_1,g_3]=-z_3e_6\in\mathcal{F}$, hence $z_3=0$. Then also $z_5\neq 0$, since otherwise $g_3$ is a multiple of $e_6\in\mathcal{F}$. 

The basis is now of the form
\begin{equation*}
g_1=e_2+x_3e_3+x_6e_6,\quad g_2=e_4-e_3+y_6e_6,\quad g_3=e_5+z_6e_6.
\end{equation*}
Considering the symplectic form $\omega$ as in Lemma \ref{lem:614symp} and the fact that $\omega_{26}\neq 0$ according to Lemma \ref{lem:sub L614}, it follows that $\omega(g_2,g_3)\neq 0$, hence $\mathcal{G}$ cannot be Lagrangian. This is a contradiction.

In the second case $\mathcal{F}$ has a basis of the form
\begin{equation*}
e_1+a_2e_2+a_3e_3+a_4e_4,\quad e_5,\quad e_6.
\end{equation*}
It follows that $\mathcal{G}$ has a basis of the form
\begin{equation*}
g_1=e_2+x_5e_5+x_6e_6,\quad g_2=e_3+y_5e_5+y_6e_6,\quad g_3=e_4+z_5e_5+z_6e_6.
\end{equation*}
Then $[g_1,g_2]=-e_6\in\mathcal{F}$. This is a contradiction. 
\end{proof}

\subsubsection{$L_{6,15}$}
\begin{lem}\label{lem:symp 615}
A $2$-form $\omega$ on $L_{6,15}$ is closed if and only if it is of the form
\begin{align*}
\omega&=\omega_{12}\alpha_{12}+\omega_{13}\alpha_{13}+\omega_{23}\alpha_{23}+\omega_{14}\alpha_{14}+\omega_{24}\alpha_{24}-\omega_{15}\alpha_{34}\\
&\quad+\omega_{15}\alpha_{15}-\omega_{16}\alpha_{35}+\omega_{16}\alpha_{16}.
\end{align*}
If the form $\omega$ is symplectic, then $\omega_{16}\neq 0$.
\end{lem}

\begin{lem}\label{lem:sub L615} Suppose that $L_{6,15}$ has two complementary $3$-dimensional subalgebras $\mathcal{F},\mathcal{G}$. Then one of the two subalgebras, say $\mathcal{F}$, has a basis that contains two vectors of the following form:
\begin{equation*}
\textstyle e_1+\sum_{i=2}^5a_ie_i,\quad e_6.
\end{equation*}
\end{lem}
\begin{proof}
Suppose that $\mathcal{F}$ has basis vectors $f_1,f_2$ as in Lemma \ref{lem:lie algebra subspace sum} (a). We get
\begin{align}
[f_1,f_2]&=-(b_2e_4+(b_4+a_2b_3-a_3b_2)e_5+(b_5-a_3b_4+a_4b_3)e_6)\label{eqn:L615}\\
[f_1,[f_1,f_2]]&=b_2e_5+(b_4+a_2b_3-2a_3b_2)e_6\nonumber.
\end{align}
Lemma \ref{lem:subalg} implies that $b_2=0$ and $b_4+a_2b_3-2a_3b_2=0$. There are two cases:
\begin{enumerate}
\item $b_5-a_3b_4+a_4b_3\neq 0$. Then equation \eqref{eqn:L615} shows that $e_6\in\mathcal{F}$.
\item $b_5-a_3b_4+a_4b_3=0$. A third basis vector of $\mathcal{F}$ is then of the form
\begin{equation*}
f_3=c_3e_3+c_4e_4+c_5e_5+c_6e_6
\end{equation*} 
with $c_4+a_2c_3-2a_3c_2=0$. There are then two subcases:
\begin{itemize}
\item $b_3=0$ and $c_3=0$. Then also $b_4=b_5=0$, hence only $b_6\neq 0$ and $e_6\in\mathcal{F}$.
\item $b_3\neq 0$ and $c_3=0$. Then also $c_4=0$. If $c_5\neq 0$, then the equation $[f_1,f_3]=-c_5e_6$ shows that $e_6\in\mathcal{F}$. This case cannot occur, because then $\mathcal{F}$ had to be at least $4$-dimensional. If $c_5=0$, then only $c_6\neq 0$ and again $e_6\in\mathcal{F}$. 
\end{itemize}
\end{enumerate}
\end{proof}

\begin{prop} The Lie algebra $L_{6,15}$ does not admit a bi-Lagrangian structure.
\end{prop}
\begin{proof} 
Considering a symplectic form as in Lemma \ref{lem:symp 615} it follows that the subalgebra $\mathcal{F}$ in Lemma \ref{lem:sub L615} cannot be Lagrangian.
\end{proof}

\subsubsection{$L_{6,17}^+$}

\begin{lem}\label{lem:617+} Let $\mathcal{F}$ be a $3$-dimensional subalgebra of $L_{6,17}^+$ with a basis vector having a non-zero $e_1$-component. Then $\mathcal{F}$ has a basis of either one of the following two forms:
\begin{enumerate}
\item $e_1+a_2e_2+a_3e_3+a_5e_5,\quad e_4,\quad e_6$
\item $e_1+a_2e_2+a_3e_3+a_4e_4,\quad e_5+b_4e_4,\quad e_6$
\end{enumerate}
\end{lem}
\begin{proof}
Suppose that $\mathcal{F}$ has basis vectors $f_1,f_2$ as in Lemma \ref{lem:lie algebra subspace sum} (a). We get
\begin{align}
[f_1,f_2]&=-(b_2e_3+b_3e_4+(a_2b_3-a_3b_2)e_5+(b_4+a_2b_5-a_5b_2)e_6)\label{eq:L617+}\\
[f_1,[f_1,f_2]]&= b_2e_4+a_2b_2e_5+(b_3(1+a_2^2)-a_2a_3b_2)e_6.\nonumber
\end{align}
Lemma \ref{lem:subalg} implies that $b_2=b_3=0$. There are two cases:
\begin{enumerate}
\item $b_5=0$. If also $b_4=0$, then only $b_6\neq 0$ and $e_6\in\mathcal{F}$. This results in a basis of type (a) or (b). If $b_4\neq 0$, then equation \eqref{eq:L617+} shows that $e_6\in\mathcal{F}$. This results in a basis of type (a).
\item $b_5\neq 0$. A third basis vector of $\mathcal{F}$ is then of the form
\begin{equation*}
f_3=c_4e_4+c_6e_6.
\end{equation*}
If $c_4=0$, then $c_6\neq 0$ and $e_6\in\mathcal{F}$. This results in a basis of type (b). If $c_4\neq 0$, then the equation $[f_1,f_3]=-c_4e_6$ again implies that $e_6\in\mathcal{F}$. This case cannot occur, because then $\mathcal{F}$ had to be at least $4$-dimensional.
\end{enumerate}
\end{proof}
\begin{prop}
The Lie algebra $L_{6,17}^+$ does not have two complementary $3$-dimensional subalgebras $\mathcal{F},\mathcal{G}$. In particular, $L_{6,17}^+$ does not admit a bi-Lagrangian structure.
\end{prop}
\begin{proof} Suppose that $L_{6,17}^+$ has two complementary $3$-dimensional subalgebras $\mathcal{F},\mathcal{G}$. Considering the bases in Lemma \ref{lem:617+}, it follows that exactly one of $\mathcal{F},\mathcal{G}$ must have a basis with a non-zero $e_1$-component, say $\mathcal{F}$. Otherwise $\mathcal{F}$ and $\mathcal{G}$ cannot be complementary. 

In the first case $\mathcal{F}$ has a basis of the form
\begin{equation*}
e_1+a_2e_2+a_3e_3+a_5e_5,\quad e_4,\quad e_6.
\end{equation*}
It follows that $\mathcal{G}$ has a basis of the form
\begin{equation*}
g_1=e_2+x_4e_4+x_6e_6,\quad g_2=e_3+y_4e_4+y_6e_6,\quad g_3=e_5+z_4e_4+z_6e_6.
\end{equation*}
We get $[g_1,[g_1,g_2]]=e_6\in\mathcal{F}$. This is a contradiction.

In the second case $\mathcal{F}$ has a basis of the form
\begin{equation*}
e_1+a_2e_2+a_3e_3+a_4e_4,\quad e_5+b_4e_4,\quad e_6.
\end{equation*}
It follows that $\mathcal{G}$ has a basis of the form
\begin{align*}
g_1&=e_2+x_4e_4+x_5e_5+x_6e_6\\
g_2&=e_3+y_4e_4+y_5e_5+y_6e_6\\
g_3&=z_4e_4+z_5e_5+z_6e_6.
\end{align*}
We get $[g_1,[g_1,g_2]]=e_6\in\mathcal{F}$. This is again a contradiction.
\end{proof}

\subsubsection{$L_{6,17}^-$}
\begin{lem}\label{lem:617-symp} Every closed $2$-form on $L_{6,17}^-$ is of the form
\begin{align*}
\omega&=\omega_{12}\alpha_{12}+\omega_{13}\alpha_{13}+\omega_{23}\alpha_{23}+\omega_{14}\alpha_{14}+\omega_{15}\alpha_{24}-\omega_{26}\alpha_{34}\\
&\quad+\omega_{15}\alpha_{15}+\omega_{25}\alpha_{25}-\omega_{16}\alpha_{35}+\omega_{16}\alpha_{16}+\omega_{26}\alpha_{26}.
\end{align*}
The form $\omega$ is symplectic if and only if $(\omega_{16}^2+\omega_{26}^2)\omega_{15}-\omega_{16}\omega_{26}(\omega_{25}+\omega_{14})\neq 0$.
\end{lem}

\begin{lem}\label{lem:sub L617-} Let $\mathcal{F}$ be a $3$-dimensional subalgebra of $L_{6,17}^-$ with a basis vector having a non-zero $e_1$-component. Then $\mathcal{F}$ has a basis of either one of the following four forms:
\begin{enumerate}
\item $e_1+a_2e_2+a_3e_3+a_5e_5,\quad e_4,\quad e_6$
\item $e_1+a_2e_2+a_3e_3+a_4e_4,\quad e_5+b_4e_4,\quad e_6$
\item $e_1+e_2+a_5e_5+a_6e_6,\quad e_3+b_5e_5+b_6e_6,\quad e_4+e_5-b_5e_6$
\item $e_1-e_2+a_5'e_5+a_6'e_6,\quad e_3+b_5'e_5+b_6'e_6,\quad e_4-e_5+b_5'e_6$
\end{enumerate}
Subalgebras with bases in (c) and (d) cannot be Lagrangian for any symplectic form on $L_{6,17}^-$.
\end{lem}
\begin{proof}
Suppose that $\mathcal{F}$ has basis vectors $f_1,f_2$ as in Lemma \ref{lem:lie algebra subspace sum} (a). We get
\begin{align}
[f_1,f_2]&=-(b_2e_3+b_3e_4+(a_2b_3-a_3b_2)e_5+(b_4-a_2b_5+a_5b_2)e_6)\label{eq:L617-}\\
[f_1,[f_1,f_2]]&= b_2e_4+a_2b_2e_5+(b_3(1-a_2^2)+a_2a_3b_2)e_6.\nonumber
\end{align}
Lemma \ref{lem:subalg} implies that $b_2=0$ and $b_3(1-a_2^2)=0$. A third basis vector of $\mathcal{F}$ is then of the form
\begin{equation*}
f_3=c_3e_3+c_4e_4+c_5e_5+c_6e_6
\end{equation*}
with $c_3(1-a_2^2)=0$. There are two cases:
\begin{enumerate}
\item $b_3=0$ and $c_3=0$.
There are two subcases:
\begin{itemize}
\item $b_5=0$. If also $b_4=0$, then only $b_6\neq 0$ and $e_6\in\mathcal{F}$. This results in a basis of type (a) or (b). If $b_4\neq 0$, then equation \eqref{eq:L617-} shows that $e_6\in\mathcal{F}$. This results in a basis of type (a).
\item $b_5\neq 0$. The third basis vector of $\mathcal{F}$ is then of the form
\begin{equation*}
f_3=c_4e_4+c_6e_6.
\end{equation*}
If $c_4=0$, then $c_6\neq 0$ and $e_6\in\mathcal{F}$. This results in a basis of type (b). If $c_4\neq 0$, then the equation $[f_1,f_3]=-c_4e_6$ again implies that $e_6\in\mathcal{F}$. This case cannot occur, because then $\mathcal{F}$ had to be at least $4$-dimensional.
\end{itemize}
\item $b_3\neq 0$ and $c_3=0$. Then $a_2=\pm 1$. According to equation \eqref{eq:L617-}
\begin{equation*}
[f_1,f_2]=-b_3e_4-a_2b_3e_5-(b_4-a_2b_5)e_6.
\end{equation*}
It follows that we can assume that the third basis vector $f_3$ is $[f_1,f_2]$ up to a non-zero multiple, hence
\begin{equation*}
f_2=e_3+b_5e_5+b_6e_6,\quad f_3=e_4\pm e_5\mp b_5e_6.
\end{equation*}
This results in a basis of type (c) or (d).
\end{enumerate}
Suppose that the basis in (c) spans a Lagrangian subspace for a symplectic form as in Lemma \ref{lem:617-symp}. Then the following pairings of the first and third and second and third basis vector have to vanish:
\begin{equation*}
\omega_{14}+\omega_{15}-b_5\omega_{16}+\omega_{15}+\omega_{25}-b_5\omega_{26}=0,\quad -\omega_{26}-\omega_{16}=0,
\end{equation*}
hence
\begin{equation*}
\omega_{26}=-\omega_{16},\quad \omega_{14}+2\omega_{15}+\omega_{25}=0.
\end{equation*}
However, $\omega$ is symplectic if and only if $\omega_{16}^2(\omega_{14}+2\omega_{15}+\omega_{25})\neq 0$. This is a contradiction.

Suppose that the basis in (d) spans a Lagrangian subspace for a symplectic form as in Lemma \ref{lem:617-symp}. Then the following pairings of the first and third and second and third basis vector have to vanish:
\begin{equation*}
\omega_{14}-\omega_{15}+b_5'\omega_{16}-\omega_{15}+\omega_{25}-b_5'\omega_{26}=0,\quad -\omega_{26}+\omega_{16}=0,
\end{equation*}
hence
\begin{equation*}
\omega_{26}=\omega_{16},\quad \omega_{14}-2\omega_{15}+\omega_{25}=0.
\end{equation*}
However, $\omega$ is symplectic if and only if $\omega_{16}^2(\omega_{14}-2\omega_{15}+\omega_{25})\neq 0$. This is a contradiction.
\end{proof}

\begin{prop} The Lie algebra $L_{6,17}^-$ does not admit a bi-Lagrangian structure.
\end{prop}
\begin{proof}
Suppose that $\mathcal{F},\mathcal{G}$ is a bi-Lagrangian structure. Considering the bases (a) and (b) in Lemma \ref{lem:sub L617-}, it follows that exactly one of $\mathcal{F},\mathcal{G}$ must have a basis with a non-zero $e_1$-component, say $\mathcal{F}$. Otherwise $\mathcal{F}$ and $\mathcal{G}$ cannot be complementary. 

In the both cases $\mathcal{F}$ has a basis of the form
\begin{equation*}
e_1+a_2e_2+a_3e_3+a_4e_4+a_5e_5,\quad b_4e_4+b_5e_5,\quad e_6.
\end{equation*}
It follows that $\mathcal{G}$ has a basis of the form
\begin{align*}
g_1&=e_2+x_4e_4+x_5e_5+x_6e_6\\
g_2&=e_3+y_4e_4+x_5e_5+y_6e_6\\
g_3&=z_4e_4+z_5e_5+z_6e_6.
\end{align*}
We get $[g_1,[g_1,g_2]]=-e_6\in\mathcal{F}$. This is a contradiction.
\end{proof}

\subsubsection{$L_{6,18}$, $L_{6,19}$, $L_{6,21}$}
\begin{lem}\label{lem:6,18,19,21}
Let $\mathfrak{g}$ be one of the Lie algebras $L_{6,18}$, $L_{6,19}$, $L_{6,21}$ and suppose that $\mathfrak{g}$ has two complementary $3$-dimensional subalgebras $\mathcal{F},\mathcal{G}$. Then one of the two subalgebras, say $\mathcal{F}$, has a basis of the form
\begin{equation*}
\textstyle e_1+\sum_{i=2}^4a_ie_i,\quad e_5,\quad e_6.
\end{equation*}
\end{lem}
\begin{proof}
Suppose that $\mathcal{F}$ has basis vectors $f_1,f_2$ as in Lemma \ref{lem:lie algebra subspace sum} (a).

In $L_{6,18}$ we get $[f_1,[f_1,f_2]] = b_2e_4+b_3e_5+b_4e_6$. Lemma \ref{lem:subalg} implies that $b_2=b_3=b_4=0$, hence the claim.

In $L_{6,19}$ we get 
\begin{equation*}
[f_1,[f_1,f_2]] = b_2e_4+b_3e_5+(b_4+a_2b_2)e_6.
\end{equation*}
Again we have $b_2=b_3=b_4=0$.

Finally, in $L_{6,21}$ we get 
\begin{equation*}
[f_1,[f_1,f_2]] = b_2e_4+(b_3+a_2b_2)e_5+(b_4+2a_2b_3-a_3b_2)e_6.
\end{equation*}
Again $b_2=b_3=b_4=0$.
\end{proof}

\begin{prop}
The Lie algebras $L_{6,18}$, $L_{6,19}$ and $L_{6,21}$ do not admit a bi-Lagrangian structure.
\end{prop}
\begin{proof}
Any closed $2$-form on $L_{6,18}$ is of the form
\begin{equation*}
\omega=\omega_{12}\alpha_{12}+\omega_{13}\alpha_{13}+\omega_{23}\alpha_{23}+\omega_{14}\alpha_{14}+\omega_{34}\alpha_{34}+\omega_{15}\alpha_{15}-\omega_{34}\alpha_{25}+\omega_{16}\alpha_{16}.
\end{equation*}
Any closed $2$-form on $L_{6,19}$ is of the form
\begin{align*}
\omega&=\omega_{12}\alpha_{12}+\omega_{13}\alpha_{13}+\omega_{23}\alpha_{23}+\omega_{14}\alpha_{14}+\omega_{16}\alpha_{24}+\omega_{34}\alpha_{34}+\omega_{15}\alpha_{15}\\
&\quad-\omega_{34}\alpha_{25}+\omega_{16}\alpha_{16}.
\end{align*}
Finally, any closed $2$-form on $L_{6,21}$ is of the form 
\begin{align*}
\omega&=\omega_{12}\alpha_{12}+\omega_{13}\alpha_{13}+\omega_{23}\alpha_{23}+\omega_{14}\alpha_{14}+\omega_{24}\alpha_{24}+\omega_{34}\alpha_{34}+\omega_{24}\alpha_{15}\\
&\quad+(\omega_{16}-\omega_{34})\alpha_{25}+\omega_{16}\alpha_{16}.
\end{align*}
If the form $\omega$ is symplectic, then in each case $\omega_{16}\neq 0$. It follows that the subalgebra $\mathcal{F}$ in Lemma \ref{lem:6,18,19,21} cannot be Lagrangian.
\end{proof}

\subsection{Calculation of the curvature of the bi-Lagrangian structures in Table \ref{table:symp}}

We now calculate the canonical connection and the curvature for the examples of bi-Lagrangian structures in Table \ref{table:symp}, see Table \ref{table:curvature} for a summary. It turns out that all $16$ examples are Ricci-flat (and thus yield para-K\"ahler analogues of Calabi--Yau manifolds), $8$ of which are flat and $8$ non-flat.

\begin{rem}
The calculations of the canonical connection can be checked with the statement in Remark \ref{rem:canonical charact}.
\end{rem}

\subsubsection{$A_6$}
A simple calculation shows that the canonical connection $\nabla$ is trivial in the basis $e_1,e_2,e_3,e_4,e_5,e_6$ and $R=0$. 

\subsubsection{$L_3\oplus A_3$}
\begin{equation*}
\nabla_{e_1}e_1=-e_3,\quad \nabla_{e_1}e_2=-e_6.
\end{equation*}
\begin{equation*}
R=0.
\end{equation*}

\subsubsection{$L_{5,2}\oplus A_1$}
\begin{equation*}
\nabla_{e_1}e_1=-e_4,\quad \nabla_{e_3}e_3=e_5,\quad \nabla_{e_3}e_1=e_6,\quad \nabla_{e_1}e_2=-e_5.
\end{equation*}
\begin{equation*}
R=0.
\end{equation*}

\subsubsection{$L_3\oplus L_3$}
\begin{equation*}
\nabla_{e_1}e_1=-e_4,\quad \nabla_{e_3}e_3=e_2,\quad \nabla_{e_3}e_4=-e_6,\quad \nabla_{e_1}e_2=-e_5.
\end{equation*}
\begin{equation*}
R(e_1,e_3)e_3=-e_5,\quad R(e_3,e_1)e_1=e_6.
\end{equation*}
\begin{equation*}
\mathrm{Ric}=0.
\end{equation*}

\subsubsection{$L_{6,1}$}
\begin{equation*}
\nabla_{e_1}e_1=e_2,\quad \nabla_{e_2}e_1=e_5,\quad \nabla_{e_1}e_3=-e_6,\quad \nabla_{e_2}e_4=-e_6.
\end{equation*}
\begin{equation*}
R=0.
\end{equation*}

\subsubsection{$L_{6,2}$}
\begin{align*}
\nabla_{e_1}e_1&=e_3,\quad \nabla_{e_3}e_1=e_5,\quad \nabla_{e_1+e_2}(e_1+e_2)=-2e_4\\
\nabla_{e_4}(e_1+e_2)&=-(e_5-e_6),\quad \nabla_{e_4}e_1=e_5,\quad \nabla_{e_1+e_2}e_3=-2e_5\\
\nabla_{e_3}(e_1+e_2)&=-(e_5-e_6),\quad \nabla_{e_1}e_4=e_5-e_6.
\end{align*}
\begin{equation*}
R(e_1,e_1+e_2)(e_1+e_2)=-2(e_5-e_6),\quad R(e_1+e_2,e_1)e_1=-2e_5.
\end{equation*}
\begin{equation*}
\mathrm{Ric}=0.
\end{equation*}

\subsubsection{$L_{5,3}\oplus A_1$}
\begin{align*}
\nabla_{e_1}e_1&=-e_3,\quad \nabla_{e_1}e_3=e_4,\quad \nabla_{e_3}e_1=e_4\\
\nabla_{e_1}e_2&=-e_5,\quad \nabla_{e_1}e_5=-e_6,\quad \nabla_{e_3}e_2=e_6.
\end{align*}
\begin{equation*}
R=0.
\end{equation*}

\subsubsection{$L_{6,4}$}
\begin{equation*}
\nabla_{e_1}e_2=-2e_4,\quad \nabla_{e_2}e_1=-e_4,\quad \nabla_{e_1}e_3=-e_5,\quad \nabla_{e_2}e_3=-e_6.
\end{equation*}
\begin{equation*}
R=0.
\end{equation*}

\subsubsection{$L_{6,5}$}
\begin{equation*}
\nabla_{e_1}e_3=-e_5,\quad \nabla_{e_1}e_1=e_3,\quad \nabla_{e_1}e_2=-e_4,\quad \nabla_{e_1}e_4=-e_6.
\end{equation*}
\begin{equation*}
R=0.
\end{equation*}

\subsubsection{$L_{6,6}$}
\begin{align*}
\nabla_{e_1-e_2}e_3&=-e_5,\quad \nabla_{e_1-e_2}(e_1-e_2)=-e_3,\quad \nabla_{e_1+e_5}(e_1+e_5)=-e_4\\
\nabla_{e_1+e_5}e_3&=-e_5,\quad \nabla_{e_1-e_2}(e_1+e_5)=-e_4,\quad \nabla_{e_1-e_2}e_4=e_6.
\end{align*}
\begin{equation*}
R(e_1+e_5,e_1-e_2)(e_1-e_2)=e_5,\quad R(e_1-e_2,e_1+e_5)(e_1+e_5)=-e_6.
\end{equation*}
\begin{equation*}
\mathrm{Ric}=0.
\end{equation*}

\subsubsection{$L_{6,9}$}
\begin{align*}
\nabla_{e_1}e_3&=-\frac{1}{2}e_5,\quad \nabla_{e_1}e_1=e_3,\quad \nabla_{e_3}e_1=\frac{1}{2}e_5\\
\nabla_{e_3}e_2&=e_6,\quad \nabla_{e_1}e_2=-e_4,\quad \nabla_{e_1}e_4=-e_6.
\end{align*}
\begin{equation*}
R=0.
\end{equation*}

\subsubsection{$L_{6,10}$}
\begin{align*}
\nabla_{e_4}e_1&=e_5,\quad \nabla_{e_2-e_4}(e_3+e_5)=-e_6,\quad \nabla_{e_2-e_4}(e_2-e_4)=-e_6\\
\nabla_{e_2-e_4}e_1&=e_4-e_5,\quad \nabla_{e_4}(e_2-e_4)=e_6.
\end{align*}
\begin{equation*}
R(e_2-e_4,e_1)e_1=-e_5,\quad R(e_1,e_2-e_4)(e_2-e_4)=e_6.
\end{equation*}
\begin{equation*}
\mathrm{Ric}=0.
\end{equation*}

\subsubsection{$L_{6,11}$}
\begin{align*}
\nabla_{e_1}e_1&=e_4,\quad \nabla_{e_4}e_1=e_5,\quad \nabla_{e_1}(e_3+e_5)=-e_6\\
\nabla_{e_1+e_2-e_4}(e_3+e_5)&=-e_6,\quad \nabla_{e_1+e_2-e_4}e_1=e_4-e_5,\quad \nabla_{e_1+e_2-e_4}e_4=-e_5\\
\nabla_{e_4}(e_1+e_2-e_4)&=e_6,\quad \nabla_{e_1+e_2-e_4}(e_1+e_2-e_4)=-(e_3+e_5)-e_6.
\end{align*}
\begin{equation*}
R(e_1+e_2-e_4,e_1)e_1=-2e_5,\quad R(e_1,e_1+e_2-e_4)(e_1+e_2-e_4)=2e_6.
\end{equation*}
\begin{equation*}
\mathrm{Ric}=0.
\end{equation*}

\subsubsection{$L_{6,12}$}
\begin{align*}
\nabla_{e_2}e_3&=-7e_5,\quad \nabla_{e_2}e_2=-\frac{6}{7}e_3,\quad \nabla_{e_3}e_2=-6e_5\\
\nabla_{e_2-2e_1}(e_3-e_4)&=\frac{7}{3}(-3e_5+e_6),\quad \nabla_{e_2-2e_1}(e_2-2e_1)=-\frac{10}{7}(e_3-e_4)\\
\nabla_{e_3-e_4}(e_2-2e_1)&=\frac{4}{3}(-3e_5+e_6),\quad \nabla_{e_2-2e_1}e_2=2e_3\\
\nabla_{e_3-e_4}e_2&=4e_5,\quad \nabla_{e_2-2e_1}e_3=5e_5,\quad \nabla_{e_2}(e_2-2e_1)=2(e_3-e_4)\\
\nabla_{e_3}(e_2-2e_1)&=-2(-3e_5+e_6),\quad \nabla_{e_2}(e_3-e_4)=-(-3e_5+e_6).
\end{align*}
\begin{align*}
R(e_2-2e_1,e_2)e_2&=\frac{208}{7}e_5\\
R(e_2,e_2-2e_1)(e_2-2e_1)&=-\frac{208}{21}(-3e_5+e_6).\\
\mathrm{Ric}&=0.
\end{align*}

\subsubsection{$L_{5,4}\oplus A_1$}
\begin{equation*}
\nabla_{e_1}e_5=-e_6,\quad \nabla_{e_4}e_4=-e_3,\quad \nabla_{e_4}e_1=e_5,\quad \nabla_{e_1}e_2=-e_4.
\end{equation*}
\begin{align*}
R(e_4,e_1)e_1&=e_6,\quad R(e_1,e_4)e_2=-e_3\\
R(e_2,e_1)e_1&=-e_5,\quad R(e_1,e_2)e_4=-e_3.
\end{align*}
\begin{equation*}
\mathrm{Ric}=0.
\end{equation*}

\subsubsection{$L_{6,16}$}
\begin{align*}
\nabla_{e_1}e_4&=-e_6,\quad \nabla_{e_3}e_3=-e_5,\quad \nabla_{e_3}e_2=e_5\\
\nabla_{e_3}e_1&=e_4,\quad \nabla_{e_1}e_2=-e_3.
\end{align*}
\begin{align*}
R(e_3,e_1)e_1&=e_6,\quad R(e_1,e_3)e_2=-e_5,\quad R(e_1,e_2)e_3=-e_5\\
R(e_1,e_2)e_2&=e_5,\quad R(e_2,e_1)e_1=-e_4.
\end{align*}
\begin{equation*}
\mathrm{Ric}=0.
\end{equation*}

\section{Appendix}

\begin{table}[!htbp]
\setlength{\tabcolsep}{0.3cm}
\renewcommand{\arraystretch}{1.5}
  \begin{tabular}{|c|cc|c|c|} \hline
	{\bf Algebraic structure} & \multicolumn{2}{c|}{{\bf Structure constants}} & {\bf Symplectic form} & {\bf Bi-Lagrangian structure}\\ \hline
	& $d\alpha_3$ & $d\alpha_4$ & & \\ \hline \hline
	$A_4$ & $0$ & $0$ & $\alpha_{12}+\alpha_{34}$ & $\{e_1,e_3\},\{e_2,e_4\}$ \\ \hline
	$L_3\oplus A_1$ ($\mathfrak{nil}_3\oplus\mathbb{R}$) & $0$ & $\alpha_{12}$ & $\alpha_{14}+\alpha_{23}$ & $\{e_1,e_3\},\{e_2,e_4\}$ \\ \hline
	$L_4$ ($\mathfrak{nil}_4$) & $\alpha_{12}$ & $\alpha_{13}$ & $\alpha_{14}+\alpha_{23}$  & not bi-Lagrangian \\ \hline
  \end{tabular}
	\caption{$4$-dimensional nilpotent Lie algebras ($d\alpha_i=0$ for $i=1,2$)}
\label{table:4-dim}
\end{table}

\begin{table}[!htbp]
\setlength{\tabcolsep}{0.3cm}
\renewcommand{\arraystretch}{1.5}
  \begin{tabular}{|cc|ccc|} \hline
	\multicolumn{2}{|c|}{{\bf Betti numbers}} & \multicolumn{3}{c|}{{\bf Algebraic structure}} \\ \hline
	$b_1$ & $b_2$ & Salamon & Khakimdjanov et al. & Bazzoni--Mu\~noz \\ \hline \hline
	$6$ & $15$ & $(0,0,0,0,0,0)$  & $26$ & $A_6$ \\ \hline
	$5$ & $11$ & $(0,0,0,0,0,12)$ & $25$ & $L_3\oplus A_3$ \\ \hline
	$4$ & $9$ & $(0,0,0,0,12,13)$ & $23$ & $L_{5,2}\oplus A_1$ \\ \hline
	$4$ & $8$ & $(0,0,0,0,12,34)$ & $24$ & $L_3\oplus L_3$ \\ \hline
	$4$ & $8$ & $(0,0,0,0,12,14+23)$ & $17$ & $L_{6,1}$ \\ \hline
	$4$ & $8$ & $(0,0,0,0,13+42,14+23)$ & $16$ & $L_{6,2}$ \\ \hline
	$4$ & $7$ & $(0,0,0,0,12,15)$ & $22$ & $L_4\oplus A_2$ \\ \hline
	$4$ & $7$ & $(0,0,0,0,12,14+25)$ & $21$ & $L_{5,3}\oplus A_1$ \\ \hline
	$3$ & $8$ & $(0,0,0,12,13,23)$ & $18$ & $L_{6,4}$ \\ \hline
	$3$ & $6$ & $(0,0,0,12,13,14)$ & $14$ & $L_{6,5}$ \\ \hline
	$3$ & $6$ & $(0,0,0,12,13,24)$ & $15$ & $L_{6,6}$ \\ \hline
	$3$ & $6$ & $(0,0,0,12,13,14+23)$ & $13$ & $L_{6,9}$ \\ \hline
	$3$ & $5$ & $(0,0,0,12,13+14,24)$ & $11$ & $L_{6,10}$ \\ \hline
	$3$ & $5$ & $(0,0,0,12,14,13+42)$ & $10$ & $L_{6,11}$ \\ \hline
	$3$ & $5$ & $(0,0,0,12,13+42,14+23)$ & $12$ & $L_{6,12}$ \\ \hline
	$3$ & $5$ & $(0,0,0,12,14,15)$ & $19$ & $L_{5,4}\oplus A_1$ \\ \hline
	$3$ & $5$ & $(0,0,0,12,14,15+23)$ & $9$ & $L_{6,13}$ \\ \hline
	$3$ & $5$ & $(0,0,0,12,14,15+24)$ & $20$ & $L_{5,6}\oplus A_1$ \\ \hline
	$3$ & $5$ & $(0,0,0,12,14,15+23+24)$ & $7$ & $L_{6,14}$ \\ \hline
	$3$ & $4$ & $(0,0,0,12,14-23,15+34)$ & $8$ & $L_{6,15}$ \\ \hline
	$2$ & $4$ & $(0,0,12,13,23,14)$ & $6$ & $L_{6,16}$ \\ \hline
	$2$ & $4$ & $(0,0,12,13,23,14+25)$ & $4$ & $L_{6,17}^+$ \\ \hline
	$2$ & $4$ & $(0,0,12,13,23,14-25)$ & $5$ & $L_{6,17}^-$ \\ \hline
	$2$ & $3$ & $(0,0,12,13,14,15)$ & $3$ & $L_{6,18}$ \\ \hline
	$2$ & $3$ & $(0,0,12,13,14,23+15)$ & $2$ & $L_{6,19}$ \\ \hline
	$2$ & $3$ & $(0,0,12,13,14+23,24+15)$ & $1$ & $L_{6,21}$ \\ \hline
  \end{tabular}
	\caption{Comparison of notation for $6$-dimensional symplectic nilpotent Lie algebras in \cite{S}, \cite{KGM} and \cite{BM}}
\label{table:compare}
\end{table}

\begin{table}[!htbp]
\setlength{\tabcolsep}{0.3cm}
\renewcommand{\arraystretch}{1.5}
  \begin{tabular}{|c|cccc|} \hline
	{\bf Algebraic structure} & \multicolumn{4}{c|}{{\bf Structure constants}}\\ \hline
	& $d\alpha_3$ & $d\alpha_4$ & $d\alpha_5$ & $d\alpha_6$  \\ \hline \hline
	$A_6$ & $0$ & $0$ & $0$ & $0$\\ \hline
	$L_3\oplus A_3$ & $0$ & $0$ & $0$ & $\alpha_{12}$\\ \hline
	$L_{5,2}\oplus A_1$ & $0$ & $0$ & $\alpha_{12}$ & $\alpha_{13}$ \\ \hline
	$L_3\oplus L_3$ & $0$ & $0$ & $\alpha_{12}$ & $\alpha_{34}$\\ \hline
	$L_{6,1}$ & $0$ & $0$ & $\alpha_{12}$ & $\alpha_{13}+\alpha_{24}$ \\ \hline
	$L_{6,2}$ & $0$ & $0$ & $\alpha_{13}-\alpha_{24}$ & $\alpha_{14}+\alpha_{23}$  \\ \hline
	$L_4\oplus A_2$ & $0$ & $0$ & $\alpha_{12}$ & $\alpha_{15}$   \\ \hline
	$L_{5,3}\oplus A_1$ & $0$ & $0$ & $\alpha_{12}$ & $\alpha_{15}+\alpha_{23}$  \\ \hline
	$L_{6,4}$ & $0$ & $\alpha_{12}$ & $\alpha_{13}$ & $\alpha_{23}$  \\ \hline
	$L_{6,5}$ & $0$ & $\alpha_{12}$ & $\alpha_{13}$ & $\alpha_{14}$  \\ \hline
	$L_{6,6}$ & $0$ & $\alpha_{12}$ & $\alpha_{13}$ & $\alpha_{24}$ \\ \hline
	$L_{6,9}$ & $0$ & $\alpha_{12}$ & $\alpha_{13}$ & $\alpha_{14}+\alpha_{23}$  \\ \hline
	$L_{6,10}$ & $0$ & $\alpha_{12}$ & $\alpha_{14}$ & $\alpha_{23}+\alpha_{24}$  \\ \hline
	$L_{6,11}$ & $0$ & $\alpha_{12}$ & $\alpha_{14}$ & $\alpha_{13}+\alpha_{24}$  \\ \hline
  $L_{6,12}$ & $0$ & $\alpha_{12}$ & $\alpha_{14}+\alpha_{23}$ & $\alpha_{13}-\alpha_{24}$  \\ \hline
	$L_{5,4}\oplus A_1$ & $0$ & $\alpha_{12}$ & $\alpha_{14}$ & $\alpha_{15}$ \\ \hline
	$L_{6,13}$ & $0$ & $\alpha_{12}$ & $\alpha_{14}$ & $\alpha_{15}+\alpha_{23}$  \\ \hline
	$L_{5,6}\oplus A_1$ & $0$ & $\alpha_{12}$ & $\alpha_{14}$ & $\alpha_{15}+\alpha_{24}$ \\ \hline
	$L_{6,14}$ & $0$ & $\alpha_{12}$ & $\alpha_{14}$ & $\alpha_{15}+\alpha_{23}+\alpha_{24}$   \\ \hline
	$L_{6,15}$ & $0$ & $\alpha_{12}$ & $\alpha_{14}+\alpha_{23}$ & $\alpha_{15}-\alpha_{34}$  \\ \hline
	$L_{6,16}$  & $\alpha_{12}$ & $\alpha_{13}$  & $\alpha_{23}$ & $\alpha_{14}$  \\ \hline
	$L_{6,17}^+$  & $\alpha_{12}$ & $\alpha_{13}$ & $\alpha_{23}$ & $\alpha_{14}+\alpha_{25}$ \\ \hline
	$L_{6,17}^-$ & $\alpha_{12}$ & $\alpha_{13}$ & $\alpha_{23}$ & $\alpha_{14}-\alpha_{25}$ \\ \hline
	$L_{6,18}$ & $\alpha_{12}$ & $\alpha_{13}$ & $\alpha_{14}$ & $\alpha_{15}$ \\ \hline
	$L_{6,19}$ & $\alpha_{12}$ & $\alpha_{13}$ & $\alpha_{14}$ & $\alpha_{15}+\alpha_{23}$ \\ \hline
	$L_{6,21}$ & $\alpha_{12}$ & $\alpha_{13}$ & $\alpha_{14}+\alpha_{23}$ & $\alpha_{15}+\alpha_{24}$ \\ \hline
  \end{tabular}
	\caption{Structure constants ($d\alpha_i=0$ for $i=1,2$)}
\label{table:structure}
\end{table}

\begin{table}[!htbp]
\setlength{\tabcolsep}{0.3cm}
\renewcommand{\arraystretch}{1.5}
  \begin{tabular}{|c|c|c|} \hline
	{\bf Algebraic structure} & {\bf Symplectic form} & {\bf Bi-Lagrangian structure}\\ \hline \hline
	$A_6$ & $\alpha_{12}+\alpha_{34}+\alpha_{56}$ (BM) & $\{e_1,e_3,e_5\},\{e_2,e_4,e_6\}$\\ \hline
	$L_3\oplus A_3$ & $\alpha_{16}+\alpha_{23}+\alpha_{45}$ (BM) & $\{e_1,e_3,e_4\},\{e_2,e_5,e_6\}$ \\ \hline
	$L_{5,2}\oplus A_1$ & $\alpha_{15}+\alpha_{24}+\alpha_{36}$ (BM) & $\{e_1,e_4,e_6\},\{e_2,e_3,e_5\}$ \\ \hline
	$L_3\oplus L_3$ & $\alpha_{15}+\alpha_{24}+\alpha_{36}$ (BM) & $\{e_1,e_4,e_6\},\{e_2,e_3,e_5\}$ \\ \hline
	$L_{6,1}$ & $\alpha_{16}+\alpha_{23}-\alpha_{45}$ (K) & $\{e_1,e_2,e_5\},\{e_3,e_4,e_6\}$\\ \hline
	$L_{6,2}$ & $\alpha_{16}+\alpha_{25}+\alpha_{34}$ (BM) & $\{e_1,e_3,e_5\},\{e_1+e_2,e_4,e_5-e_6\}$ \\ \hline
	$L_4\oplus A_2$ & $\alpha_{16}+\alpha_{25}+\alpha_{34}$ (BM) & not bi-Lagrangian  \\ \hline
	$L_{5,3}\oplus A_1$ & $\alpha_{16}+\alpha_{24}-\alpha_{35}$ (BM) & $\{e_1,e_3,e_4\},\{e_2,e_5,e_6\}$ \\ \hline
	$L_{6,4}$ & $\alpha_{16}+2\alpha_{25}+\alpha_{34}$ (K) & $\{e_1,e_2,e_4\},\{e_3,e_5,e_6\}$ \\ \hline
	$L_{6,5}$ & $\alpha_{16}+\alpha_{25}+\alpha_{34}$ (K) & $\{e_1,e_3,e_5\},\{e_2,e_4,e_6\}$ \\ \hline
	$L_{6,6}$ & $\alpha_{15}+\alpha_{25}-\alpha_{26}+\alpha_{34}$ (K) & $\{e_1-e_2,e_3,e_5\},\{e_1+e_5,e_4,e_6\}$ \\ \hline
	$L_{6,9}$ & $\alpha_{16}+2\alpha_{25}+\alpha_{34}$ (BM) & $\{e_1,e_3,e_5\},\{e_2,e_4,e_6\}$ \\ \hline
	$L_{6,10}$ & $\alpha_{16}+\alpha_{25}-\alpha_{34}$ (BM) & $\{e_1,e_4,e_5\},\{e_2-e_4,e_3+e_5,e_6\}$ \\ \hline
	$L_{6,11}$ & $\alpha_{16}+\alpha_{25}-\alpha_{26}-\alpha_{34}$ (K) & $\{e_1,e_4,e_5\},\{e_1+e_2-e_4,e_3+e_5,e_6\}$  \\ \hline
  $L_{6,12}$ & $-\alpha_{15}+6\alpha_{26}+7\alpha_{34}$ (K) & $\{e_2,e_3,e_5\},\{e_2-2e_1,e_3-e_4,-3e_5+e_6\}$ \\ \hline
	$L_{5,4}\oplus A_1$ & $\alpha_{13}+\alpha_{26}-\alpha_{45}$ (BM) & $\{e_1,e_5,e_6\},\{e_2,e_3,e_4\}$ \\ \hline
	$L_{6,13}$ & $\alpha_{13}+\alpha_{26}-\alpha_{45}$ (BM) & not bi-Lagrangian \\ \hline
	$L_{5,6}\oplus A_1$ & $\alpha_{13}+\alpha_{26}-\alpha_{45}$ (BM) & not bi-Lagrangian \\ \hline
	$L_{6,14}$ & $\alpha_{13}+\alpha_{26}-\alpha_{45}$ (BM) & not bi-Lagrangian \\ \hline
	$L_{6,15}$ & $\alpha_{16}+\alpha_{24}-\alpha_{35}$ (K) &  not bi-Lagrangian \\ \hline
	$L_{6,16}$  & $\alpha_{15}+\alpha_{24}+\alpha_{26} -\alpha_{34}$ (K) & $\{e_1,e_4,e_6\},\{e_2,e_3,e_5\}$ \\ \hline
	$L_{6,17}^+$  & $\alpha_{16}+\alpha_{15}+\alpha_{24} + \alpha_{35}$ (BM) & not bi-Lagrangian \\ \hline
	$L_{6,17}^-$ & $\alpha_{15}-\alpha_{16}+\alpha_{24}+\alpha_{35}$ (K) & not bi-Lagrangian \\ \hline
	$L_{6,18}$ & $\alpha_{16}+\alpha_{25}-\alpha_{34}$ (BM) &  not bi-Lagrangian \\ \hline
	$L_{6,19}$ & $\alpha_{16}+\alpha_{24}+\alpha_{25}-\alpha_{34}$ (BM) & not bi-Lagrangian \\ \hline
	$L_{6,21}$ & $2\alpha_{16}+\alpha_{25}+\alpha_{34}$ (BM) &  not bi-Lagrangian \\ \hline
  \end{tabular}
	\caption{Symplectic and bi-Lagrangian structures ((BM) and (K) indicate that the symplectic form is taken from \cite{BM} and \cite{KGM}, respectively; note Remark \ref{rem:errors})}
\label{table:symp}
\end{table}

\begin{table}[!htbp]
\setlength{\tabcolsep}{0.3cm}
\renewcommand{\arraystretch}{1.5}
  \begin{tabular}{|c|c|} \hline
	{\bf Algebraic structure} & {\bf Curvature tensor $R$}\\ \hline\hline
	$A_6$ & $0$ \\ \hline
	$L_3\oplus A_3$ & $0$ \\ \hline
	$L_{5,2}\oplus A_1$ & $0$ \\ \hline
	$L_3\oplus L_3$ & $\neq 0$ \\ \hline
	$L_{6,1}$ & $0$   \\ \hline
	$L_{6,2}$ & $\neq 0$ \\ \hline
	$L_{5,3}\oplus A_1$ & $0$   \\ \hline
	$L_{6,4}$ & $0$   \\ \hline
	$L_{6,5}$ & $0$  \\ \hline
	$L_{6,6}$ & $\neq 0$ \\ \hline
	$L_{6,9}$ & $0$    \\ \hline
	$L_{6,10}$ & $\neq 0$   \\ \hline
	$L_{6,11}$ & $\neq 0$  \\ \hline
  $L_{6,12}$ & $\neq 0$ \\ \hline
	$L_{5,4}\oplus A_1$ & $\neq 0$  \\ \hline
	$L_{6,16}$  & $\neq 0$  \\ \hline
  \end{tabular}
	\caption{Curvature of bi-Lagrangian structures in Table \ref{table:symp} (all examples are Ricci-flat)}
\label{table:curvature}
\end{table}

\clearpage

\subsection*{Acknowledgements} My work on this topic dates back to my diploma thesis \cite{MHdiplom} at the LMU Munich. I am grateful to D.~Kotschick for supervising the thesis and for many helpful discussions and suggestions in the meantime. I also want to thank the anonymous referee for helpful comments.

\bibliographystyle{amsplain}

\begin{thebibliography}{999}

\bibitem{Bai} C.~Bai, {\em Left-symmetric bialgebras and an analogue of the classical Yang-Baxter equation}, 
Commun.~Contemp.~Math.~{\bf 10} (2008), no.~2, 221--260.

\bibitem{BC} O.~Baues, V.~Cort\'es, {\sl Symplectic Lie groups}, Ast\'erisque {\bf 379}, Soci\'et\'e Math\'ematique de France 2016.

\bibitem{BM} G.~Bazzoni, V.~Mu\~noz, {\em Classification of minimal algebras over any field up to dimension 6}, Trans.~Amer.~Math.~Soc.~{\bf 364} (2012), no.~2, 1007--1028.

\bibitem{BB} S.~Benayadi, M.~Boucetta, {\em On para-K\"ahler and hyper-para-K\"ahler Lie algebras}, J.~Algebra {\bf 436} (2015), 61--101.

\bibitem{Bo} N.~B.~Boyom, {\em Varietes symplectiques affines}, Manuscripta Math.~{\bf 64} (1989), 1--33.

\bibitem{CR} D.~Conti, F.~A.~Rossi, {\em The Ricci tensor of almost parahermitian manifolds}, Ann.~Global Anal.~Geom.~{\bf 53} (2018), no.~4, 467--501. 

\bibitem{CS} D.~Conti, S.~Salamon, {\em Generalized Killing spinors in dimension 5}, Trans.~Amer.~Math.~Soc.~{\bf 359} (2007), no.~11, 5319--5343.

\bibitem{CFG} V.~Cruceanu, P.~Fortuny, P.~M.~Gadea, {\em A survey on paracomplex geometry}, Rocky Mountain J.~Math.~{\bf 26} (1996), no.~1, 83--115.

\bibitem{ES} F.~Etayo Gordejuela, R.~Santamar\'ia, {\em The canonical connection of a bi-Lagrangian manifold}, Journ.~of Phys.~A {\bf 34} (2001), 981--987.

\bibitem{G} D.~Giovannini, {\sl Special structures and symplectic geometry}, Ph.D.~thesis, Universit\`a degli Studi di Torino, 2003.

\bibitem{MHdiplom} M.~J.~D.~Hamilton, {\sl Bi-Lagrangian structures on closed manifolds}, Diplomarbeit, LMU M\"unchen 2004, 84 pp.

\bibitem{HK} M.~J.~D.~Hamilton, D.~Kotschick, {\sl K\"unneth geometry}, book manuscript.

\bibitem{He} H.~Hess, {\em Connections on symplectic manifolds and geometric quantization}, Lect.~Notes in Math.~{\bf 836} (1980), 153--166.

\bibitem{Hit} N.~J.~Hitchin, {\em Hypersymplectic quotients}, in {\sl La ''M\'ecanique analytique'' de Lagrange et son h\'eritage}, Supplemento al numero {\bf 124} (1990) degli Atti della Accademia delle Scienze di Torino, Classe de Scienze Fisiche, Matematiche e Naturali, 169--180.

\bibitem{Ka} S.~Kaneyuki, {\em Homogeneous symplectic manifolds and dipolarizations in Lie algebras}, Tokyo J.~Math.~{\bf 15} (1992), no.~2, 313--325. 

\bibitem{KGM} Y.~Khakimdjanov, M.~Goze, A.~Medina, {\em Symplectic or contact structures on Lie groups}, Differential Geom.~Appl.~{\bf 21} (2004), no.~1, 41--54.

\bibitem{Mal} A.~I.~Malcev, {\em On a class of homogeneous spaces}, AMS Transl.~{\bf 9} (Series 1) (1962) 276--307.

\bibitem{R} F.~A.~Rossi, {\sl {\bf D}-complex structures: cohomological properties and deformations}, Ph.D.~thesis, Universit\`a degli Studi di Milano--Bicocca (2013). http://hdl.handle.net/10281/41976

\bibitem{S} S.~M.~Salamon, {\em Complex structures on nilpotent Lie algebras}, J.~Pure Appl.~Algebra {\bf 157} (2001), no.~2--3, 311--333.

\end{thebibliography}

\bigskip
\bigskip

\end{document}